\documentclass[10pt]{amsart}
\usepackage{amsmath,amsthm,amssymb,curves,url}
\include{xy} 
\xyoption{all}
\SelectTips{cm}{}

\usepackage{oldgerm}

\newcommand\Z{{\mathbb Z}}

\newcommand\Znn{\Z_{\ge0}}

\newcommand\Zp{\Z_p}

\newcommand\Q{{\mathbb Q}}
\newcommand\Qpos{\Q_{>0}}

\newcommand\Qp{\Q_p}
\newcommand\Qpx{\Qp^\times}

\newcommand\R{{\mathbb R}}

\newcommand\C{{\mathbb C}}

\newcommand\SetR{{\mathcal R}}

\newcommand\SetT{{\mathcal T}}
\newcommand\SetU{{\mathcal U}}

\newcommand\lra{\longrightarrow}
\newcommand\supparen[1]{^{(#1)}}
\newcommand\suppn{\supparen n}
\newcommand\suppnminusone{\supparen {n-1}}
\newcommand\suppm{\supparen m}
\newcommand\supptwon{\supparen{2n}}

\newcommand\bs{\backslash}
\newcommand\trans{'} 
\newcommand\inv{^{-1}}
\newcommand\tr{\operatorname{tr}}
\newcommand\trsp{\operatorname{tr}\,} 
\newcommand\smallhalf{{\textstyle\frac12}}

\newcommand\ie{{\em i.e.}}

\newcommand\Mats[2]{{\rm M}_{#1}(#2)}

\newcommand\MatsnZ{\Mats n\Z}

\newcommand\Matssym[2]{{\rm M}_{#1}(#2)^{\rm sym}}
\newcommand\MatsnRsym{\Matssym n\R}
\newcommand\MatsnZsym{\Matssym n\Z}

\newcommand\QxpmW{\Q[x^{\pm1}]^W} 

\newcommand\Vsp{{\mathcal V}}
\newcommand\Vn{\Vsp_n}
\newcommand\VnZ{\Vn(\Z)}

\newcommand\ip[2]{\langle #1,#2\rangle} 

\newcommand\Pos{{\mathcal C}}
\newcommand\Pn{\Pos_{\!n}}
\newcommand\Pncl{\overline\Pos_{\!n}}

\newcommand\siX[1]{{\mathcal X}_{#1}} 
\newcommand\Xn{\siX n}
\newcommand\Xnsemi{\siX n^{\rm semi}} 
\newcommand\Xtwonsemi{\siX{2n}^{\rm semi}} 

\newcommand\Xm{\siX m}
\newcommand\Xsix{\siX 6}

\newcommand\SL[2]{\operatorname{SL}_{#1}(#2)}
\newcommand\SLtwoZ{\SL2\Z}

\newcommand\GL[2]{\operatorname{GL}_{#1}(#2)}
\newcommand\GLnZ{\GL n\Z}
\newcommand\GLnZp{\GL n{\Zp}}
\newcommand\GLmZp{\GL m{\Zp}}
\newcommand\GLnZtwo{\GL n{\Z_2}}

\newcommand\Spgp{\operatorname{Sp}}
\newcommand\Sp[2]{\operatorname{Sp}_{#1}(#2)}

\newcommand\SpnR{\Sp n\R}
\newcommand\SpnZ{\Sp n\Z}

\newcommand\Gamzero{\Gamma_{\!0}}
\newcommand\Gamone{\Gamma_{\!1}}
\newcommand\Gamtwo{\Gamma_{\!2}}
\newcommand\Gamthree{\Gamma_{\!3}}

\newcommand\Gamn{\Gamma_{\!n}}
\newcommand\Gamtwon{\Gamma_{\!2n}}
\newcommand\Gamnmone{\Gamma_{n-1}}

\newcommand\GSpgp{\operatorname{GSp}}
\newcommand\GSppos[2]{{\rm GSp}_{#1}^+(#2)}

\newcommand\GSpnposQ{\GSppos n\Q}
\newcommand\GSpnposZpinv{\GSppos n{\Z[1/p]}}

\newcommand\HA{{\mathcal H}} 

\newcommand\smallmat[4]{\left[\begin{smallmatrix}
{#1}&{#2}\\{#3}&{#4}\end{smallmatrix}\right]}

\newcommand\smallmatabcd{\smallmat abcd}

\newcommand\UHP{{\mathcal H}} 
\newcommand\UHPn{{\UHP_n}} 

\newcommand\wtvar[2]{[#1]_{#2}}
\newcommand\wtk[1]{\wtvar{#1}k}

\newcommand\fc[2]{a(#1;#2)} 
\newcommand\e{{\rm e}} 

\newcommand\Enk{{E_k\suppn}}
\newcommand\Enminusonek{{E_k\suppnminusone}}
\newcommand\Etwonk{{E_k\supptwon}}
\newcommand\Esixk{{E_k\supparen6}}
\newcommand\Eonek{{E_k\supparen1}}
\newcommand\calE{{\mathcal E}}
\newcommand\calEnk{{\calE_k\suppn}}
\newcommand\calEtwonk{{\calE_k\supptwon}}

\newcommand\MFsnoweight{{\mathcal M}} 

\newcommand\MFswtgp[2]{\MFsnoweight_{#1}(#2)}
\newcommand\MFs[1]{\MFswtgp k{#1}} 
\newcommand\CFsnoweight{{\mathcal S}} 

\newcommand\CFswtgp[2]{\CFsnoweight_{#1}(#2)}
\newcommand\CFs[1]{\CFswtgp k{#1}} 

\newcommand\mymod{\text{ mod }}

\newcommand{\Kfield}{{\mathbb K}}
\newcommand{\algints}{{\mathcal O}}
\newcommand\OK{\algints_\Kfield}
\newcommand\primep{{\mathfrak p}}

\newcommand\hlf{\textstyle{\frac12}}

\renewcommand{\Re}[1]{{\rm Re}(#1)}
\renewcommand{\Im}[1]{{\rm Im}(#1)}

\newcommand\iotast{\iota^*}
\newcommand\iotastf{\iota^*\!f}

\newcommand{\ord}[2]{{\rm ord}_{#1}(#2)}

\newcommand\Qst[3]{Q_{#1}^{\rm st}(#2,#3)}
\newcommand\Qpst[2]{\Qst p{#1}{#2}}

\newcommand\Lst[2]{L^{\rm st}(#1,#2)}
\newcommand\Qsp[3]{Q_{#1}^{\rm spin}(#2,#3)}
\newcommand\Qpsp[2]{\Qsp p{#1}{#2}}

\def\blfootnote{\gdef\@thefnmark{}\@footnotetext}

\theoremstyle{plain}
\newtheorem{theorem}{Theorem}[section]

\newtheorem{proposition}[theorem]{Proposition}

\begin{document}
\title[Katsurada's determination of Eisenstein series]
{Using Katsurada's determination of the Eisenstein series to compute
Siegel eigenforms}

\author[O.~D.~King]{Oliver~D.~King
}
\address{Dept.~of Mathematics, University of California, Berkeley, CA 94720 USA
\footnote{Current address: Department of Cell and Developmental Biology, University of Massachusetts Medical School, Worcester MA 01655 USA}
}
\email{Oliver.King@umassmed.edu} 
%
%

\author[C.~Poor]{Cris Poor}
\address{Dept.~of Mathematics, Fordham University, Bronx, NY 10458 USA}
\email{poor@fordham.edu}

\author[J.~Shurman]{Jerry Shurman}
\address{Reed College, Portland, OR 97202 USA}
\email{jerry@reed.edu}

\author[D.~Yuen]{David S.~Yuen}
\address{Dept.~of Mathematics and Computer Science, Lake Forest
College, 555 N.~Sheridan Rd., Lake Forest, IL 60045 USA}
\email{yuen@lakeforest.edu}

\date{\today}

\begin{abstract}
We compute Hecke eigenform bases of spaces of level one, degree~three Siegel
modular forms and $2$-Euler factors of the eigenforms through weight $22$.
Our method uses the Fourier coefficients of Siegel Eisenstein series,
which are fully known and computationally tract\-able 
by the work of H.\ Katsurada; we also use 
P.~Garrett's decomposition of the pullback of the Eisenstein
series through the Witt map.
Our results support I.~Miyawaki's conjectural lift,
and they give examples of eigenforms that are congruence neighbors. 

\end{abstract}



\maketitle
\section{Introduction\label{secI}}

Eisenstein series are central in the theory of Siegel 
modular forms.  The algorithmic computation of the Fourier coefficients 
of Siegel Eisenstein series began with C.~L.~Siegel and was completed by H.~Katsurada, 
whose work deserves to be widely known. 
For any positive integer degree~$n$ and even integer weight $k>n+1$,
the Siegel Eisenstein series of weight~$k$ and degree~$n$ is
$$
\Enk(z)=\sum_{\gamma\in P_\Z\bs\Gamn}j(\gamma,z)^{-k}.
$$
Here $z$ lies in the Siegel upper half space~$\UHPn$,
and the summand $j(\gamma,z)^{-k}$ is~$1$ for the Siegel parabolic
subgroup $P_\Z=\{\smallmat ab0d\}$ of the integral symplectic group
$\Gamn=\SpnZ$.
(Section~\ref{I:TN} will review the background for this paper.)
This Eisenstein series has the Fourier series representation
$$
\Enk(z)=\sum_{t\in\Xnsemi}\fc t{\Enk}\,\e(\ip tz),
$$
where $\Xnsemi$ denotes the set of semi-integral positive semidefinite
$n$-by-$n$ matrices.
The Siegel $\Phi$ map takes Eisenstein series to Eisenstein series,
$\Phi\Enk=\Enminusonek$ and $\Phi\Eonek=1$,
so it suffices to compute the Fourier coefficients of Eisenstein
series for definite indices~$t$; the set of such matrices is denoted~$\Xn$.
The Fourier coefficient formula for definite indices,
to be elaborated below, is
$$
\fc t\Enk=
\dfrac{2^{\lfloor \frac{n+1}{2} \rfloor} \prod_{p}F_p(t,p^{k-n-1})}
{\zeta(1-k)\prod_{i=1}^{\lfloor n/2\rfloor}\zeta(1-2k+2i)}
\cdot\begin{cases}
L(\chi_{D_t},1-k+n/2)&\text{$n$ even},\\
1&\text{$n$ odd}.
\end{cases}
$$
The Fourier coefficient depends only on the genus of its index~$t$.
In fact the polynomial $F_p(t,X)\in\Z[X]$ depends only on the class
of~$t$ over~$\Zp$. 
Algorithmic specification of these $F_p$ polynomials was the last
impediment to evaluating Siegel Eisenstein series Fourier coefficients,
and it was overcome by recursion relations due to Katsurada
\cite{katsurada99}. 
This article explains how the Fourier coefficients of Siegel
Eisenstein series are computationally accessible, and it makes
computer programs to evaluate them publicly available. 

For example, consider the Fourier coefficient index
\begin{equation*}
t=\frac12\left[\begin{matrix}
2 & 1 & 1 & 0 & 1 & 2 \\
1 & 4 & 2 & 2 & 0 & 1 \\
1 & 2 & 4 & 2 & 0 & 0 \\
0 & 2 & 2 & 4 & 2 & 2 \\
1 & 0 & 0 & 2 & 4 & 2 \\
2 & 1 & 0 & 2 & 2 & 8 
\end{matrix}\right]\in\Xsix.
\end{equation*}
Our genus symbol program takes~$2t$ as an input and returns the genus
symbol $4^{-2}_4\, 3^{-1}$. 
Our $F_p$ polynomial program takes this genus symbol and the
determinant $\det(2t)=48$ as input and returns the $F_p(t,X)$
polynomials for all~$p\mid2\det(2t)$,
$$
\left[F_2(X),F_3(X)\right]=\left[1+24X+256X^2+3072X^3+16384X^4,1\right].
$$
These data make no reference to any particular Eisenstein series
degree or weight. With these $F_p$ polynomials and the weight $k=16$
as input, our Eisenstein series Fourier coefficient program returns
$$
\fc {t}{E^{(6)}_{16}}
=\frac{ 9780154654408147370255260881715200}
{13912726954911229324966739363569}\,. 
$$
We hope that our programs \cite{yuen15} will be useful to researchers. 

\medskip

This article describes our computation of the $2$-Euler factors of
Hecke eigenform bases of the degree~$3$ cusp form spaces $\CFs\Gamthree$
for even~$k$ up through~$22$.
Along with Katsurada's completion of the Siegel Eisenstein series Fourier
coefficient formula, which makes a coefficient readily computable from
a genus symbol of its index, the second idea of our method is
P.~Garrett's decomposition of the pulled back Eisenstein series
\cite{garrett84}. Thus we call our algorithm the pullback--genus method.

Our computation naturally continues the computational work in
\cite{miyawaki92}, where I.~Miyawaki studied the weights $k=12,14$
and conjectured two kinds of lift in consequence of his results.
Miyawaki's first lifting conjecture with an additional natural
nonvanishing condition has been settled by T.~Ikeda \cite{ikeda06},
who gave a construction in general degree that generalizes the
degree~$3$ case; we call these lifts Ikeda--Miyawaki lifts. 
Miyawaki's second lifting conjecture is still open. Our computations
in weights $16$ through~$22$ support the conjecture, exhibiting
eigenforms whose $2$-Euler factors agree with all of Miyawaki's
predicted lifts of the second type. Some of our results were
announced in \cite{ikpy14}, which forward referenced this article.

Besides the Ikeda-Miyawaki lifts and the conjectural Miyawaki lifts, 
we see nonlift eigenforms that have unimodular Satake parameters at~$2$. 
So far, each nonlift eigenform has at least one lift eigenform as a 
{\em congruence neighbor\/}. 
Two eigenforms $f$ and~$g$ are congruence neighbors if all their
eigenvalues are congruent modulo some prime ideal~$\primep$ 
in the ring of integers of the field of eigenvalues, and the
ideal~$\primep$ is  a {\em congruence prime\/}.
Thus, although not all eigenforms are lifts, through weight~$22$ all
eigenforms are lifts or congruence neighbors of lifts. 


We summarize our results here, but the reader is encouraged to examine
the database \cite{yuen15}, which includes Fourier coefficients, Euler
factors, noncusp forms, and more. 
\begin{itemize}
\item $\dim(\CFswtgp{16}\Gamthree)=3$,
with a conjugate pair $M^{\rm I}_{16}(f_{28,\pm},g_{16})$ of
Ikeda--Miyawaki lifts over the quadratic number field of
the $\CFswtgp{28}\SLtwoZ$ Hecke eigenbasis,
and one apparent nonlift over~$\Q$ with unimodular Satake parameters at~$2$.
The apparent nonlift is a congruence neighbor of the Ikeda--Miyawaki
lifts modulo a prime over~$107$.
This result was announced in \cite{ikpy14}.
\item $\dim(\CFswtgp{18}\Gamthree)=4$,
with a conjugate pair $M^{\rm I}_{18}(f_{32,\pm},g_{18})$ of
Ikeda--Miyawaki lifts over the quadratic number field of the
$\CFswtgp{32}\SLtwoZ$ Hecke eigenbasis,
and a conjugate pair $M^{\rm II}_{18}(f_{34,\pm},g_{16})$
of apparent Miyawaki~II lifts over the quadratic number field of the
$\CFswtgp{34}\SLtwoZ$ Hecke eigenbasis.
This space has no congruence neighbors for large primes.
\item $\dim(\CFswtgp{20}\Gamthree)=6$,
with a conjugate triple $M^{\rm I}_{20}(f_{36,(1,2,3)},g_{20})$ of
Ikeda--Miyawaki lifts over the cubic number field of the
$\CFswtgp{36}\SLtwoZ$ Hecke eigenbasis,
a conjugate pair $M^{\rm II}_{20}(f_{38,\pm},g_{18})$ of apparent
Miyawaki~II lifts over the quadratic number field of the
$\CFswtgp{38}\SLtwoZ$ Hecke eigenbasis,
and one apparent nonlift over~$\Q$ with unimodular Satake parameters at~$2$.
The apparent nonlift is a congruence neighbor of the Ikeda--Miyawaki
lifts modulo a prime over~$157$. 
This result was announced in \cite{ikpy14}.
\item $\dim(\CFswtgp{22}\Gamthree)=9$,
with a conjugate triple $M^{\rm I}_{22}(f_{40,(1,2,3)},g_{22})$ of
Ikeda--Miyawaki lifts over the cubic number field of the
$\CFswtgp{40}\SLtwoZ$ Hecke eigenbasis,
a conjugate triple of apparent Miyawaki~II lifts
$M^{\rm II}_{22}(f_{42,(1,2,3)},g_{20})$
over the cubic number field of the $\CFswtgp{42}\SLtwoZ$ Hecke eigenbasis,
and a conjugate triple of apparent nonlifts over a real cubic field with unimodular Satake
parameters at~$2$.
The apparent nonlifts are congruence neighbors of the
Ikeda--Miyawaki lifts modulo a prime over $1753$, and they are
congruence neighbors of the apparent Miyawaki~II lifts modulo two
primes, one over $613$ and the other over $677$.
\end{itemize}
Heim \cite{heim10} has raised the question of whether
$\operatorname{GL}(2)$-twisted $L$-functions would appear as spinor
$L$-functions of degree~three Siegel cusp eigenforms. Our computations
show that these $L$-functions do not arise in $\CFs\Gamthree$
for~$k\le22$, which suggests 
that Miyawaki found all the lifts 
that occur in degree three and level one. 


\medskip

Section~\ref{I:TN} reviews the setting for our computations. 
Section~\ref{secEFC} reviews the Eisenstein series Fourier coefficient formula.
Section~\ref{secPG} describes the pullback--genus method,
which gives a rational basis of $\MFs{\Gamma_n}$.
Section~\ref{secHEE} reviews how to compute Hecke eigenform Euler factors.
Section~\ref{secImp} explains the pullback-genus method's implementation, 
which is essential for making these computations tractable.

We thank S.~B\"ocherer, P.~Garrett, B.~Heim, T.~Ibukiyama, and
H.~Katsurada for their support for this project. 
We thank Fordham University for the use of the Jacobi and ACE servers.


\section{Background\label{I:TN}}

Let $n$ be a positive integer.
Let $\Vn$ denote the vector space $\MatsnRsym$ of symmetric
$n\times n$ real matrices, carrying the inner product $\ip tu=\tr(tu)$.
Let $\Pn$ denote the cone of positive definite elements of $\Vn$. The
Siegel upper half space of degree~$n$ is
$\UHPn=\{z=x+iy:x\in\Vn,\ y\in\Pn\}$.
The symplectic group $\Spgp(n)$ is defined by the $2n\times2n$
matrix condition $J[g]=J$ where $J=\smallmat0{-1}1{\phantom{-}0}$ with
each block $n\times n$, and $J[g]=g\trans Jg$ with $g\trans$ the
transpose of~$g$.
The real symplectic group acts on the Siegel space,
$g\langle z \rangle=(az+b)(cz+d)\inv$ where $g=\smallmatabcd$.
For any $g\in\SpnR$ and any $z\in\UHPn$ the factor of automorphy
$j(g,z)$ is $\det(cz+d)$. For any integer~$k$, the weight~$k$
action $\wtk{\,\cdot\,}$ of~$\SpnR$ on functions $f:\UHPn\lra\C$ is
$f\wtk g(z)=j(g,z)^{-k}f(g\langle z \rangle)$.

The integral symplectic group $\SpnZ$ is denoted~$\Gamn$.
For any integer~$k$, a function $f:\UHPn\lra\C$ is a
Siegel modular form of weight $k$ and degree $n$ if (1) $f$ is
holomorphic, (2) $f\wtk\gamma=f$ for all $\gamma\in\Gamn$,
(3) for any $y_o\in\Pn$, $f$ is bounded on the set $\{z\in\UHPn:
\Im z>y_o\}$.
The vector space of such Siegel modular forms is denoted $\MFs\Gamn$.
For any positive integer~$n$, Siegel's map
$\Phi:\MFs\Gamn\lra\MFs{\Gamnmone}$ is given by
$(\Phi f)(z)=\lim_{\eta\to+\infty}f(\smallmat z00{i\eta})$.
The kernel of Siegel's $\Phi$-map on~$\MFs\Gamn$
is the vector space of Siegel cusp forms of degree~$n$ and weight~$k$,
denoted $\CFs\Gamn$.
By convention $\MFs{\Gamzero}=\CFs{\Gamzero}=\C$.

The Euclidean space $\Vn$ contains the integer lattice
$\VnZ=\Vn\cap\MatsnZ=\MatsnZsym$, whose dual lattice $\VnZ^*$ is
all $v\in\Vn$ such that $v_{ii}\in\Z$ and $v_{ij}\in\smallhalf\Z$ for
all~$i,j$.
Let $\Xn=\VnZ^*\cap\Pn$ and $\Xnsemi=\VnZ^*\cap\Pncl$ denote the
positive definite elements and the positive semidefinite elements of
the dual lattice.
Any Siegel modular form has a Fourier series representation, 
setting $e(x) = e^{2\pi i x}$, 
$$
f(z)=\sum_{t\in\Xnsemi}\fc tf\,\e(\ip tz),
$$
and $f$ is a cusp form if and only its Fourier series is supported
on~$t\in\Xn$.

As already discussed, the Siegel parabolic subgroup of
$\Gamn$-elements $\smallmat ab0d$ is denoted~$P_\Z$, and the Siegel
Eisenstein series of degree~$n\ge1$ and even weight~$k>n+1$ is
$\Enk(z)=\sum_{\gamma\in P_\Z\bs\Gamn}j(\gamma,z)^{-k}$ for $z\in\UHPn$,
and the Fourier coefficients $\fc t{\Enk}$ of the Eisenstein series
have been fairly well understood since C.~L.~Siegel's work of~1939
\cite{siegel39} but their practical description for arbitrary~$n$ was
completed by Katsurada only in 1999 \cite{katsurada99}.
In general the Fourier coefficients of a Siegel modular form~$f$ of even
weight are $\GLnZ$-equivalence class functions,
meaning that $\fc\cdot f$ is constant over each class $t[\GLnZ]$.
But a key computational point here is that, given any decomposition
$t\sim u\oplus0_{n-m}$ where $u\in\Xm$ is strictly positive, the
Eisenstein series Fourier coefficient $\fc t\Enk$ is determined by
only the genus of~$u$, \ie, by the set of matrices in~$\Xm$ that lie in
$u[\GLmZp]$ for every prime~$p$. The genus is a coarser equivalence
class than $t[\GLnZ]$, and a symbol for it is much faster to compute
than the $\GLnZ$-class.

The symplectic similitude group $\GSpgp(n)$ is defined by the
condition $J[g]=m(g)J$ for some invertible multiplier~$m(g)$.
The rational symplectic positive similitude group $\GSpnposQ$,
carrying the condition $m(g)\in\Qpos$, acts on~$\UHPn$ via 
$f\wtk g(z)=m(g)^e j(g,z)^{-k}f(g\langle z \rangle)$,
where the classical choice of the multiplier power is $e=kn-\langle n\rangle$
with $\langle n\rangle=n(n+1)/2$.
Any double coset in $\Gamn\bs\GSpnposQ/\Gamn$ decomposes as finitely
many right cosets, $\Gamn g\Gamn=\bigsqcup_{i=1}^d\Gamn g_i$, and it
acts correspondingly on~$\MFs\Gamn$ by 
$f\wtk{\Gamn g\Gamn}(z)=\sum_{i=1}^df\wtk{g_i}(z)$.
The double cosets generate a commutative algebra over~$\Q$, the Hecke
algebra $\HA(\Gamn,\GSpnposQ)$.
The space $\MFs\Gamn$ has a basis of Hecke eigenforms.
Using standard generators of the Hecke algebra, 
any $f\in\MFs\Gamn$ is a Hecke eigenform if for each prime it is an
eigenform of $T(p)=\Gamn{\rm diag}(1_n,p1_n)\Gamn$
and of $T_i(p^2)=\Gamn{\rm diag}(1_{n-i},p1_i,p^21_{n-i},p1_i)\Gamn$
for $i=1,\cdots,n-1$.

Let $z_1\oplus z_2=\smallmat{z_1}00{z_2}$ for $z_1,z_2\in\UHPn$.
The symplectic embedding $\iota(z_1\times z_2)= z_1\oplus z_2$ pulls back
to a map of functions, $(\iotastf)(z_1\times z_2)=f(z_1\oplus z_2)$.
The pullback~$\iotast$ takes $\MFs\Gamtwon$ to $\MFs\Gamn\otimes\MFs\Gamn$
and takes $\CFs{\Gamtwon}$ to $\CFs{\Gamn}\otimes\CFs{\Gamn}$ 
by results of E.~Witt \cite{witt41}. 
These maps and variants of them are often called the Witt map.
Garrett's formula \cite{garrett84}
(originally in a 1979 preprint by Garrett, then in S.~B\"ocherer's Ph.D.
directed by H.~Klingen and in a paper of M.~Harris)
decomposes the pulled back Eisenstein series $\iotast\Etwonk$
as a sum of nonzero multiples of the diagonal tensor products
over a Hecke eigenform basis $\{f_1,\cdots,f_d\}$ of~$\MFs\Gamn$,
$$
\iotast\Etwonk=\sum_{\ell=1}^dc_\ell\,f_\ell\otimes f_\ell,
\quad\text{all $c_\ell$ nonzero}.
$$
That is, $\Etwonk(z_1\oplus z_2)=\sum_\ell c_\ell
f_\ell(z_1)f_\ell(z_2)$ for $z_1,z_2\in\UHPn$.
This connection between the Hecke eigenform basis and the Eisenstein
series is what guarantees that the
pullback--genus method works. 
However, when the dimension of $\MFs{\Gamma_n}$ is already known, 
the computations can be rigorously executed using only the existence
of the Witt map. 

Let $\Q[x^{\pm1}]$ denote the algebra of rational Laurent polynomials
in indeterminates $x_0,x_1,\cdots,x_n$.
The Weyl group~$W$ of this algebra is generated by the permutations of
$\{x_1,\cdots,x_n\}$ and the involutions $\tau_i$ for~$i=1,\cdots,n$
taking $x_0$ to~$x_0x_i$ and $x_i$ to~$x_i\inv$.
Fix a prime~$p$, and let $\HA_p=\HA(\Gamn,\GSpnposZpinv)$.
The Satake isomorphism $\Omega=\Omega_p$ from $\HA_p$ to the
subalgebra $\QxpmW$ of Laurent polynomials invariant under the Weyl
group is defined on any double coset $\Gamma g\Gamma$ by taking each
of its constituent cosets $\Gamma b$, with Borel subgroup representative
$b=\smallmat{p^{e_0}d^*}*0d$ whose $d$-block has diagonal
$(p^{e_1},\cdots,p^{e_n})$, to $\Omega(\Gamma b)
=x_0^{e_0}(x_1/p)^{e_1}(x_2/p^2)^{e_2}\cdots(x_n/p^n)^{e_n}$.
For any Siegel Hecke eigenform~$f$, the eigenvalue homomorphism
$\lambda_f:\HA_p\lra\C$ is defined by the condition $Tf=\lambda_f(T)f$
for $T\in\HA_p$. There exists a Satake parameter
$\alpha=\alpha_{f,p}\in\C^{n+1}$ such that
$\lambda_f(T)=(\Omega(T))(\alpha_{f,p})$ for all~$T\in\HA_p$.
The standard $L$-function of a Siegel Hecke eigenform~$f$ is the
product of Euler factors defined in terms of the Satake parameters.
Specifically, for each prime~$p$ let the $p$th Satake parameter be
$\alpha_{p,f}=(\alpha_{0,p},\alpha_{1,p},\cdots,\alpha_{n,p})$;
then $\Lst fs=\prod_p\Qpst f{p^{-s}}\inv$
where $\Qpst fX=(1-X)\prod_{i=1}^n(1-\alpha_{i,p}X)(1-\alpha_{i,p}^{-1}X)$.

Miyawaki \cite{miyawaki92} computed for the generator~$F_{12}$
of~$\CFswtgp{12}\Gamthree$ that $\Lst{F_{12}}s$ has the same $2$-Euler
factor as $L(f_{20},s+10)L(f_{20},s+9)\Lst{g_{12}}s$
where $f_{20}$ and $g_{12}$ respectively generate
$\CFswtgp{20}\SLtwoZ$ and $\CFswtgp{12}\SLtwoZ$,
and he computed for the generator~$F_{14}$ of~$\CFswtgp{14}\Gamthree$
that $\Lst{F_{14}}s$ has the same $2$-Euler factor as
$L(f_{26},s+13)L(f_{26},s+12)\Lst{g_{12}}s$.
He conjectured that for any even weight~$k$, and for each pair of elliptic
Hecke eigenforms $f\in\CFswtgp{2k-4}\SLtwoZ$ and $g\in\CFswtgp{k}\SLtwoZ$, 
there exists a Siegel Hecke eigenform $F\in\CFs\Gamthree$
whose standard $L$-function factors as
$\Lst Fs= L(f,s+k-2)L(f,s+k-3)\Lst gs$, and for each pair of elliptic
Hecke eigenforms $f\in\CFswtgp{2k-2}\SLtwoZ$ and $g\in\CFswtgp{k-2}\SLtwoZ$
there exists a Siegel Hecke eigenform $F\in\CFs\Gamthree$
whose standard $L$-function factors as
$\Lst Fs= L(f,s+k-1)L(f,s+k-2)\Lst gs$.
Granting a nonvanishing condition, Ikeda \cite{ikeda06} established a
general lift subsuming Miyawaki's first conjectured lift. The Hecke
eigenfunction behavior of this lift was shown by Ikeda, Heim
\cite{heim12}, and Hayashida \cite{hayashida14}.

Let $c_k\supptwon=2^{-n}\zeta(1-k)\prod_{i=1}^n\zeta(1-2k+2i)$.
The dilated Eisenstein series $\calEtwonk=c_k\supptwon\Etwonk$
has rational Fourier coefficients that are integral at all primes $p>2k-1$.
Garrett's formula with $\calEtwonk$ in place of~$\Etwonk$
can quickly show congruences between Hecke eigenform basis elements.
Specifically, the Fourier coefficients of a Hecke eigenform can be taken to
lie in the integer ring~$\OK$ of a number field~$\Kfield$,
and the constants~$c_\ell$ in the formula to lie in~$\Kfield^\times$.
Call a maximal ideal~$\primep$ of~$\OK$ {\em big\/} if it lies over a
rational prime $p>2k-1$.
Suppose that there exists a big prime~$\primep$ of~$\OK$ such that
$\ord\primep{c_1}=\ord\primep{c_2}=-1$
and $\ord\primep{c_\ell}\ge0$ for $\ell=3,\cdots,d$.
Then for any index~$s$ such that $\ord\primep{\fc s{f_1}}=0$, we deduce 
$f_1/\fc s{f_1}=f_2/\fc s{f_2}\mymod\primep$,
and especially, if $\fc s{f_1}=\fc s{f_2}=1$ then $f_1=f_2\mymod\primep$. 
Thus $f_1$ and $f_2$ are congruence neighbors and $\primep$ is a congruence prime.



\section{Eisenstein Series Fourier Coefficients\label{secEFC}}

\subsection{$F_p$-Polynomials\label{EFC:FPP}}

Polynomials $F_p(u,X)\in\Z[X]$ for $p$ prime and $u\in\Xm$ appear in
the Siegel Eisenstein series Fourier coefficient formula. 
The first author of this paper wrote a program to compute these
polynomials \cite{king03}, which has since been modified to accept
higher degree input. We refer to \cite{katsurada99} for the
definition of the $F_p$ polynomials; there Katsurada proved a
functional equation for these polynomials, which was an important
step in his establishment of their recurrence relations. 
We review this functional equation because it serves as a check on
computations. The functional equation makes reference to the
Hilbert symbol and to the Hasse invariant. To review, for $a,b\in\Qpx$
the Hilbert symbol $(a,b)_p$ is $1$ if $aX^2+bY^2=Z^2$ has nontrivial
solutions in~$\Qp^3$ and $-1$ if not. For $u\in\GL m\Qp^{\rm sym}$
the Hasse invariant of~$u$ is $h_p(u)=\prod_{i\le j}(a_i,a_j)_p$ where
$u$ is $\GL m\Qp$-equivalent to the diagonal matrix having entries
$a_1,\cdots,a_m$.
If $m$ is even then $(-1)^{m/2}\det(2u)$ takes the form $D_uf_u^2$
where $D_u$ is~$1$ or the fundamental discriminant of a quadratic
number field and $f_u$ is a positive integer; let $\chi_{D_u}$ denote the
quadratic Dirichlet character of conductor~$|D_u|$.
For rank $m=0$, the empty matrix has determinant~$1$ by convention and
so $D_u=f_u=1$.

\begin{theorem}[Katsurada's Functional Equation\label{KatFE}]
Let $u\in\Xm$. Set
$$
e_p(u)=\begin{cases}
2(\lfloor\frac{\ord p{\det(2u)}-1-\delta_{p,2}}{2}\rfloor+\chi_{D_u}(p)^2)
&\text{if $m$ is even},\\
\ord p{\det(2u)/2}&\text{if $m$ is odd}.
\end{cases}
$$
Here $\delta_{p,2}$ is the Kronecker delta. Then
$$
F_p(u,p^{-m-1} X\inv)= \pm(p^{(m+1)/2}X)^{-e_p(u)}F_p(u,X),
$$
where if $m$ is even then the ``$\pm$'' sign is positive,
and if $m$ is odd then it is
$$
\big(\det(u),(-1)^{(m-1)/2}\det(u)\big)_p\,(-1,-1)_p^{(m^2-1)/8}\,h_p(u)
$$
with $(\cdot,\cdot)_p$ the Hilbert symbol and $h_p$ the Hasse invariant 
as described above.
\end{theorem}


\subsection{Fourier Coefficient Formula\label{EFC:FCF}}

Let $n$ be a positive integer.
For any $t\in\Xnsemi$ we have $t\sim u\oplus0_{n-m}$ under
$\GLnZ$-equivalence, where $m={\rm rank}(t)\in\Znn$ and $u\in\Xm$.
The following result may be found in \cite{katsurada99, katsurada10}. 

\begin{theorem}[Siegel Eisenstein Fourier Coefficient Formula\label{EisFCthm}]
Let $n$ be a positive integer and $k>n+1$ an even integer.
Let $t\in\Xnsemi$, and let $u$, $D_u$, $f_u$, and $\chi_{D_u}$
be as above. Let $c\suppm_k=2^{-\lfloor(m+1)/2\rfloor}
\zeta(1-k)\prod_{i=1}^{\lfloor m/2\rfloor}\zeta(1-2k+2i)$.
Then
$$
\fc t\Enk={1}/{c_k\suppm}\cdot\begin{cases}
L(\chi_{D_u},1-k+m/2)\prod_{p\mid f_u}F_p(u,p^{k-m-1}),
&\text{$m$ even},\\
\prod_{p:\ord p{(1/2)\det(2u)}>0}F_p(u,p^{k-m-1})
&\text{$m$ odd}.
\end{cases}
$$
\end{theorem}

The Riemann zeta values and the quadratic $L$ value in the formula
have the form $\zeta(1-j)=-B_j/j$ and $L(\chi,1-j)=-B_j(\chi)/j$
with the $B_j$ basic or quadratic Bernoulli numbers,
and so they are known rational numbers: 
if $f$ is the conductor of~$\chi$, then 
$\sum_{a=1}^f \chi(a) \frac{te^{at}}{e^{ft}-1} = \sum_{j=0}^{\infty}
B_j(\chi) \frac{t^j}{j!}$ (\cite{aik14}, page~{53}).
The genus symbol of any $u\in\Xm$ is easy to compute (section~\ref{PG:GS}),
and then our program gives $F_p(u,p^{k-m-1})$.
Thus Siegel Eisenstein series Fourier coefficients are tractable.

The Clausen--von Staudt theorem for basic and quadratic Bernoulli
numbers shows that the dilated Eisenstein series $\calEnk=c\suppn_k\Enk$
has rational Fourier coefficients that are integral at all primes $p>2k-1$.
While the monic Eisenstein series has the computational advantage that
its Fourier coefficients depend only on the nonsingular part of their
indices, with no reference to the degree~$n$, the integrality of the
dilated Eisenstein lets us identify congruence neighbors.

For lower weights $\lceil(n+1)/2\rceil\le k\le n+1$, excluding a few cases,
a Siegel Eisenstein series $\Enk(z,s)$ with a complex parameter~$s$
can be continued leftward from its half plane of absolute convergence
$\Re{k+2s}>n+1$ to~$s=0$, where it is a Siegel modular form in~$z$
(see, for example, the introduction to \cite{weissauer86}).
Because our work here uses weights $k>2n+1$ we do not discuss these issues.

\section{The Pullback--Genus Method\label{secPG}}

Recall Garrett's formula,
$\iotast E_k\supparen{2n}=\sum_{\ell=1}^dc_\ell\,f_\ell\otimes f_\ell$
for even~$k >2 n+1$, where $\{f_1,\cdots,f_d\}$ is a Hecke
eigenform basis of~$\MFs\Gamn$. 
Garrett's conjecture that the $c_{\ell}$ are nonzero was proven by S.~B{\"o}cherer \cite{bocherer83}, 
and this result is important in the proof of our Proposition~\ref{PGdimprop}.  
For any two indices $t_1,t_2\in\Xnsemi$, equate the $t_1\times t_2$
Fourier coefficients on the two sides of Garrett's formula to get a
relation among Fourier coefficients and the~$c_\ell$,
\begin{equation}
\label{sorry}
\sum_{r\in\SetR(t_1\times t_2)}
\fc{\smallmat{t_1}r{r\trans}{t_2}}{E_k\supparen{2n}}
=\sum_{\ell=1}^dc_\ell\,\fc{t_1}{f_\ell}\fc{t_2}{f_\ell},
\end{equation}
summing the left side over
$\SetR(t_1\times t_2)=\{r\in\Mats n{\smallhalf\Z}:
\smallmat{t_1}r{r\trans}{t_2}\in\Xtwonsemi\}$.
This set is finite because $\smallmat{t_{1ii}}{r_{ij}}{r_{ij}}{t_{2jj}}$
is positive semidefinite for all~$i,j$, bounding $r_{ij}^2$
by~$t_{1ii}t_{2jj}$.
The summand on the left side is tractable, and additionally 
the set of summation can be traversed quickly enough to make the left
side computationally accessible (section~\ref{PG:IE}).
Conceptually, for each fixed $t_1 \in \Xnsemi$, 
equation~{(\ref{sorry})} gives the Fourier coefficient at $t_2$ of 
an element in $\MFs {\Gamma_n}$ with rational Fourier coefficients. 
We use enough different~$t_1$ to obtain a rational basis. 
If we additionally want eigenforms and Euler factors, 
we apply Hecke operators to this rational basis to obtain a basis of eigenforms~$f_\ell$, 
with some convenient choice of normalization. 
%
%
Sufficiently many Fourier coefficients of a eigenform basis
of~$\CFs\Gamn$ enable us to compute Euler factors of their 
$L$-functions (section~\ref{secHEE}).
If we want to identify congruence neighbors among the $f_{\ell}$,
we then use equation~{(\ref{sorry})} to solve for the coefficients $c_{\ell}$. 

\subsection{Index Enumeration\label{PG:IE}}

Given $t_1,t_2\in\Xnsemi$, the following algorithm traverses
$\SetR(t_1\times t_2)$ quickly enough for our programs to terminate.
Immediately multiply by~$2$ to work with integers; that is, 
double the~$t_i$ for this algorithm and introduce
the matrix $s=\smallmat{t_1}r{r\trans}{t_2}$, find all integral~$r$
that make~$s$ positive semidefinite, and divide each such~$r$ by~$2$
before returning it.
The algorithm builds matrices~$s$ by filling~$r$ columnwise.
Thus, when determining possible provisional values for some~$r_{i,j}$,
a set of provisional values is already present for all other
$r$-entries having row index at most~$i$ and column index at most~$j$.

As noted, $|r_{i,j}|\le m$ where $m=\sqrt{t_{1,i,i}t_{2,j,j}}$.
A first version of the algorithm is therefore $n^2$ nested loops:
For each~$|r_{i_1,j_1}|\le m_1$,
for each~$|r_{i_2,j_2}|\le m_2$,
\dots,
for each~$|r_{i_{n^2},j_{n^2}}|\le m_{n^2}$,
test $s$ for positive semidefiniteness.
We refer to these $m_i$ as the default loop-bounds. 
The algorithm admits two refinements that cut down the nested looping
at the cost of further bounds-checking.
The refinements are most easily explained by example. Let $\tilde s$
denote a submatrix of~$s$ as in the following diagram, in which $n=3$.
$$
s=\left[\begin{array}{ccc|ccc}
\circ&\circ&\cdot&*&*&\cdot\\
\circ&\circ&\cdot&*&r_{i,j}&\cdot\\
\cdot&\cdot&\cdot&\cdot&\cdot&\cdot\\
\hline
*&*&\cdot&\circ&\circ&\cdot\\
*&r_{i,j}&\cdot&\circ&\circ&\cdot\\
\cdot&\cdot&\cdot&\cdot&\cdot&\cdot
\end{array}\right],
\qquad
\tilde s=\left[\begin{array}{cc|cc}
\circ&\circ&*&*\\
\circ&\circ&*&r_{i,j}\\
\hline
*&*&\circ&\circ\\
*&r_{i,j}&\circ&\circ
\end{array}\right].
$$
The circles are entries of $t_1$ and~$t_2$ and the asterisks are some
of the outer loop variables.
First, the loop-bounds of $r_{i,j}$ can be improved. Introduce
three auxiliary matrices~$a$, $b$, and~$c$ by setting $a$ to $\tilde s$ but with the row
and column of the two $r_{i,j}$-entries deleted,
$b$ to $\tilde s$ but with $0$ in place of the higher $r_{i,j}$
and with the row and column of the lower $r_{i,j}$ deleted,
and $c$ to $\tilde s$ but with $0$ in place of both $r_{i,j}$ entries.
Then $\det\tilde s= -\det a\cdot r_{i,j}^2
+(-1)^j2\det b\cdot r_{i,j}+\det c$.
The algorithm needs to have checked the positive semidefiniteness of
previous matrices to ensure that $\det a$ is nonnegative.
In our implementation of columnwise traversal, the check needs to
happen at the bottom of each column.
When $\det a$ is positive, the condition $\det\tilde s\ge0$ is
quadratic and yields bounds of~$r_{i,j}$, generally tighter
than the default loop-bounds.
As a second refinement, we may check whether a value of $r_{i,j}$ makes~$\tilde s$
positive semidefinite before proceeding to more inner loops.
In practice there are tradeoffs between na\"ive looping and checking
to tighten loop bounds or abort inner loops.
Our current implementation is to improve the loop bounds
and to check~$\tilde s \ge 0$ at the bottom of each column of~$r$.

\subsection{Genus Symbol\label{PG:GS}}

Consider any $t\in\Xn$.
A symbol for the genus of~$t$ is described in chapter~15, section~7
of~\cite{conway99}. We summarize it briefly.
As in the index-set traversal algorithm, immediately double~$t$
to ensure integral entries.
The finitely many equivalence classes $\{t[\GLnZp]: p\mid2\det t\}$
determine every other class $t[\GLnZp]$ where $p\nmid2\det t$.
Thus the genus symbol of~$t$ need only describe its local integral
equivalence class for each $p\mid2\det t$.

For an odd prime divisor~$p$ of~$\det t$,
$t$ is $\GLnZp$-equivalent to some
$\bigoplus_{i=1}^kp^{e_i}\delta_i$ where $0\le e_1<\cdots<e_k$
and each $\delta_i$ is a diagonal matrix having $p$-adic units on the
diagonal. Each $p^{e_i}\delta_i$ is a {\em constituent\/} of~$t$,
and $p^{e_i}$ is the {\em scale\/} of the constituent.
The $\GLnZp$-equivalence symbol of~$t$ is
$q_1^{\epsilon_1 n_1}\,q_2^{\epsilon_2 n_2}\,\cdots\,q_k^{\epsilon_k n_k}$
where for $i=1,\cdots,k$,
$q_i=p^{e_i}$ and $\epsilon_i$ is the Legendre symbol~$(\det\delta_i/p)$
and $n_i$ is the size of~$\delta_i$.
The $\GLnZp$-equivalence symbol of~$t$ is uniquely defined by~$t$,
and it completely characterizes $t[\GLnZp]$.

Also, $t$ is $\GLnZtwo$-equivalent to some
$\bigoplus_{i=1}^k2^{e_i}d_i$ where $0\le e_1<\cdots<e_k$
and each $d_i$ is either a diagonal matrix with units on the diagonal
or a direct sum $d_i=\bigoplus_{j=1}^{\ell_i}\delta_{ij}$ where
each~$\delta$ is a $2\times2$ matrix $\smallmat{2^\alpha a}bb{2^\gamma c}$.
The $\GLnZtwo$-equivalence symbol of~$t$ is
$(q_1)_{t_1}^{\epsilon_1 n_1}\,(q_2)_{t_2}^{\epsilon_2 n_2}\,\cdots\,
(q_k)_{t_k}^{\epsilon_k n_r}$,
where for $i=1,\cdots,k$,
$q_i=2^{e_i}$ and $\epsilon_i$ is the Kronecker symbol~$(\det d_i/2)$
($1$ if $\det d_i=\pm1\mymod8$, $-1$ if $\det d_i=\pm3\mymod8$)
and $n_i$ is the size of~$d_i$
and $t_i$ is~$\trsp d_i\,{\rm(mod}\ 8{\rm)}$ if $d_i$ is diagonal,
while it is undefined or~$\infty$ if $d_i$ is a sum of $2\times2$ subblocks.
The $\GLnZtwo$-equivalence symbol of~$t$ determines $t[\GLnZtwo]$,
but not conversely.
A unique symbol can be produced as explained in \cite{conway99}, but
there are computational tradeoffs between actually computing a canonical $2$-adic symbol 
and using a method to compute 
$F_p$ polynomials that accepts 
different symbols for the same genus. 
Since our program accepts any $2$-adic symbol,
we do not discuss how to make the $2$-part of the genus symbol unique.

\subsection{Determining Bases\label{PG:DB}}

The next proposition shows how to find a rational basis for {\it any\/} 
level one space of Siegel modular forms, $\MFs{\Gamma_n}$, for even $k > 2n+1$. 
It is an application of Garrett's formula and highlights the importance of the work 
of Katsurada, which makes the computation of the Fourier coefficients of 
Eisenstein series practical.

\begin{proposition}[Pullback--Genus Method\label{PGdimprop}]
Consider even $k >2 n+1$. 
Let $\SetT=\{t_1,\cdots t_m\}$ be a determining set of Fourier
coefficient indices for~$\MFs\Gamn$.
Define an $m$-by-$m$ matrix 
$M=[\fc{t_i\times t_j}{\iotast E_k\supparen{2n}}]_{m\times m}$.
Then $\dim(\MFs\Gamn)={\rm rank}(M)$.

Column reduce~$M$ to get a matrix
$\left[\begin{array}{c|c}*&0\end{array}\right]$.
The nonzero columns describe the $\SetT$-truncations of a basis
of~$\MFs\Gamn$, \ie, each nonzero column contains the $\SetT$th Fourier
coefficients of a basis element.
Further stipulating that $\SetT$ is ordered with the singular indices
at the beginning, let the column reduction of~$M$ take the form
$$
M\sim\left[\begin{array}{c|c|c}
*&0&0\\
\hline
*&*&0
\end{array}\right]
\quad\text{{\rm(}column space equivalence\/{\rm)}},
$$
with the horizontal dividing line after the rows indexed by singular~$t$.
The columns between the vertical dividing lines describe the
$\SetT$-truncations of a basis of~$\CFs\Gamn$. In particular, 
$\dim(\CFs\Gamn)$ is the number of columns between the
vertical dividing lines.
\end{proposition}

\begin{proof}
Garrett's formula gives 
$M=\sum_{\ell=1}^dc_\ell[\fc{t_i}{f_\ell}\fc{t_j}{f_\ell}]_{m\times m}$,
where $d=\dim(\MFs\Gamn)$ is the desired dimension.
Each summand matrix is an outer product $v_\ell
v_\ell\trans$ where the column vector~$v_\ell$ encodes a determining
truncation of the Fourier series of~$f_\ell$,
$M=\sum_{\ell=1}^dc_\ell v_\ell v_\ell\trans$
where $v_\ell=[\fc{t_1}{f_\ell}\quad\cdot\quad\fc{t_m}{f_\ell}]\trans$
for $\ell=1,\cdots,d$.
The matrix sum thus has the form
$M=VCV\trans$ where $V=[v_1\quad\cdots\quad v_d]_{m\times d}$
and $C={\rm diag}(c_1,\cdots,c_d)$.
Because $\SetT$ is a determining set, the $d$ columns of~$V$ are linearly
independent.
Because each~$c_\ell$ is nonzero, $C$ is invertible.
Therefore the column space of~$CV\trans$ is~$\C^d$, and consequently
$\text{col\,sp}(VCV\trans)=\text{col\,sp}(V)
={\rm span}(v_1,\cdots,v_d)$, which has dimension~$d$.
This shows that $\dim(\MFs\Gamn)={\rm rank}(M)$.
For any $j\in\{1,\cdots,m\}$, the $j$th column of~$M$ is
$\sum_{\ell=1}^dc_{j,\ell}
[\fc{t_1}{f_\ell}\quad\cdot\quad\fc{t_m}{f_\ell}]\trans$
with each $c_{j,\ell}=c_\ell\,\fc{t_j}{f_\ell}$.
This is $[\fc{t_1}{g_j}\quad\cdots\quad\fc{t_m}{g_j}]\trans$
where $g_j=\sum_{\ell=1}^dc_{j,\ell}f_\ell$.
After column reducing~$M$ to the form
$\left[\begin{array}{c|c|c}
*&0&0\\
\hline
*&*&0
\end{array}\right]
$
described in the proposition, the desired statements about the reduced
matrix are immediate since its rows are indexed by a determining set
and its rank is $\dim(\MFs\Gamn)$.
\end{proof}

We make a few comments about Proposition~\ref{PGdimprop}. 
We have methods available to obtain a finite determining set~$\SetT$ of
indices for $\MFs\Gamn$ as needed by the proposition, see \cite{pooryuen07}. 
These determining sets are not needed when the dimension is known. 
In degree three, 
Tsuyumine gave the generating function for $\dim(\MFs{\Gamthree})$
over all weights~$k$ (\cite{tsuyumine86} page~831, though the factor
$(1-T^{12})^3$ in the denominator at the top of page~832 should
be $(1-T^{12})^2$ instead).
So for~$n=3$ we do not need the full strength of Proposition~\ref{PGdimprop}.
Instead, we may simply grow the rows of~$M$ columnwise until we have a
basis, or even take a subset of the rows of~$M$ and grow that matrix
columnwise until getting a basis.

\section{Computing Euler Factors\label{secHEE}}

This section condenses section~3 of \cite{pry09}. 
The Hecke action on Fourier expansions is explained in \cite{breulmann00}. 
Thus, given a basis of $\MFs{\Gamma_n}$ with sufficiently long Fourier expansions, 
we may compute a basis $\{ f_\ell\}$ of eigenforms and their eigenvalues 
$T(p) f = \lambda_f(T(p))f$ and $T_j(p^2) f = \lambda_f(T_j(p^2)) f$.

Letting square brackets connote Weyl group symmetrization,
introduce elements of the invariant Laurent polynomial algebra $\QxpmW$,
$g=[x_0]=x_0\prod_{i=1}^n(x_i+1)$
and $g_\ell=[x_0^2x_1\cdots x_{n-\ell}]$ for $\ell=0,\cdots,n$.
Altogether $g,g_0,g_1,\cdots,g_{n-1},g_0\inv$ generate $\QxpmW$.
For any prime~$p$, the Hecke algebra generators $T(p)$, $\{T_i(p^2)\}$
and the invariant polynomial algebra generators $g$, $\{g_\ell\}$
are related via the Satake isomorphism $\Omega:\HA_p\lra\QxpmW$ and
linear relations as follows (Hilfssatz~3.14 and Hilfssatz~3.17
in~\cite{freitag83}).

\begin{proposition}[Satake Isomorphism on Generators\label{Kriegmatprop}]
Let $n\ge2$ be an integer, and let $p$ be prime.
Then $\Omega(T(p))=g$.
Also, there exists an upper triangular matrix
$K=K_n(p^2)\in\Mats{n+1}{\Z[1/p]}$, with positive entries on and
above the diagonal, such that {\rm(}applying $\Omega$ componentwise on
the left side of the next equation\/{\rm)}
$$
\Omega\left[T_n(p^2)\quad\cdots\quad T_1(p^2)\quad T_0(p^2)\right]
=\left[g_0\quad g_1\quad\cdots\quad g_n\right]K_n(p^2).
$$
\end{proposition}

A.~Krieg \cite{krieg86} gave the entries of the matrix for $n\ge2$,
and a program that computes the matrix is at the authors's website
\cite{yuen15}. In particular,
$$
K_3(p^2)=\left[\begin{matrix}
\dfrac1{p^6} & \dfrac{p^3-1}{p^6} & \dfrac{3p^3-p^2-p-1}{p^4}
& \dfrac{(p-1)(3p^3-p^2-p-1)}{p^4} \\
0 & \dfrac1{p^3} & \dfrac{p^2-1}{p^3} & \dfrac{2(p-1)}{p} \\
0 & 0 & \dfrac1{p} & \dfrac{p-1}{p} \\ 
0 & 0 & 0 & 1
\end{matrix}\right].
$$

Introduce further Weyl-invariant polynomials $r_0,r_1,\cdots,r_{2n}$,
defined by the relation
$\prod_{i=1}^n(1-x_iX)(1-x_i\inv X)=\sum_{\ell=0}^{2n}(-1)^\ell r_\ell X^\ell$,
with $r_{2n-\ell}=r_\ell$.
Thus the standard $p$-Euler factor
$\Qpst fX=(1-X)\prod_{i=1}^n(1-\alpha_{i,p}X)(1-\alpha_{i,p}^{-1}X)$
of a Hecke eigenform $f\in\MFs\Gamn$ is
$\Qpst fX=(1-X)\sum_{\ell=0}^{2n}(-1)^\ell r_\ell(\alpha)X^\ell$.
The $r_\ell$ and $g_\ell$ Laurent polynomials are related by the condition
$$
[r_0\quad r_1\quad\cdots\quad r_n]
=p^{\langle n\rangle-kn}[g_0\quad g_1\quad\cdots\quad g_n]\,P,
$$
where $P=P_n$ is the Pascal-like upper triangular matrix
whose nonzero entries are $\binom{n-i}{(j-i)/2}$ in the $(i,j)$th
position if $j-i\in2\Znn$, the row and column indices starting at~$0$.
In particular,
$$
P_3=\left[\begin{matrix}
1 & 0 & 3 & 0 \\ 0 & 1 & 0 & 2 \\ 0 & 0 & 1 & 0 \\ 0 & 0 & 0 & 1
\end{matrix}\right].
$$
Given a Hecke eigenform, if we can compute its eigenvalues under
$T(p)$ and under $T_i(p^2)$ for $i=0,\cdots,n$ then we can produce the
values $r_i(\alpha)$ that specify its standard $p$-Euler factor, as
follows \cite{pry09}.

\begin{theorem}[Standard Euler Factor From Eigenvalues\label{stdeuthm}]
Let $n\ge2$ and~$k$ be positive integers, and let $p$ be prime. 
Let $f\in\MFs\Gamn$ be a Hecke eigenform of $\HA_p$.
Let $\alpha$ be the Satake parameter of the eigenvalue function~$\lambda_f$.
Introduce the vector of eigenvalues and the vector of polynomial coefficients,
$\vec\lambda=[\lambda_f(T_n(p^2))\quad\cdots\quad\lambda_f(T_0(p^2))]$
and $\vec r=[r_0(\alpha)\quad\cdots\quad r_n(\alpha)]$.
Then $\vec r=p^{\langle n\rangle-kn}\vec\lambda\,K\inv P$,
where the matrices $P$ and~$K$ are as above.
\end{theorem}

Indeed, introducing $\vec g=[g_0(\alpha)\quad\cdots\quad g_n(\alpha)]$
we have $\vec r=p^{\langle n\rangle-kn}\vec gP$,
and the Satake mapping property and Proposition~\ref{Kriegmatprop}
combine to give $\vec g=\vec\lambda\,K\inv$.

Computing the spinor Euler factor from the eigenvalues is similar.
In degree~$3$, the spinor $p$-Euler factor is
\begin{align*}
\Qpsp fX
&=(1-\alpha_0X)
\cdot(1-\alpha_0\alpha_1X)(1-\alpha_0\alpha_2X)(1-\alpha_0\alpha_3X)\\
&\quad\cdot(1-\alpha_0\alpha_1\alpha_2X)(1-\alpha_0\alpha_1\alpha_3X)
(1-\alpha_0\alpha_2\alpha_3X)
(1-\alpha_0\alpha_1\alpha_2\alpha_3X).
\end{align*}
Denote its expansion $\sum_{\ell=0}^8(-1)^\ell s_\ell X^\ell$.
By direct computation, or \cite{miyawaki92}, page~{310}, 
\begin{alignat*}2
s_0&=1&\qquad s_8&=g_0^4(\alpha)s_0\\
s_1&=g(\alpha)&\qquad s_7&=g_0^3(\alpha)s_1\\
s_2&=(4g_0+2g_1+g_2)(\alpha)&\qquad s_6&=g_0^2(\alpha)s_2\\
s_3&=g(\alpha)(g_0+g_1)(\alpha)&\qquad s_5&=g_0(\alpha)s_3\\
s_4&=(2g_0^2+4g_0g_1+g_0g_3+g_1^2)(\alpha).
\end{alignat*}
By the Satake mapping property and Proposition~\ref{Kriegmatprop},
$g(\alpha)=\lambda_f(T(p))$ and (as above) $\vec g=\vec\lambda\,K\inv$.
The spinor factor follows from the previous display.

\section{Implementation\label{secImp}}

Implementing the pullback--genus method is not a purely mechanical matter.

For a given weight~$k$, one wants to choose a small determining set
of indices~$\mathcal T$ that aptly comprises matrices of rank $1$, $2$, and~$3$ in
light of the known dimensions of $\MFs\Gamone$, $\MFs\Gamtwo$, and
$\MFs\Gamthree$. These matrices should have small entries. But also
the space that we are trying to determine can have an element that
vanishes to high order, such as Igusa's 
$\chi_{18}\in \CFswtgp{18}\Gamthree$, requiring a bigger index to ``see'' it in
order for the method to succeed. Guessing a small determining set for a
given degree~$k$ requires a combination of software experimentation
and mathematical insight into the structure of $\MFs\Gamthree$.

The index-enumeration algorithm of section~\ref{PG:IE} is a significant bottleneck, quickly
growing expensive as the entries of the elements of the determining
set grow. To carry out our computations through weight~$22$ we
processed $1965$ Fourier coefficient indices $t_1\times t_2$
of the pullback $\iotast\Esixk$, leading to $1\,561\,537\,201$ Fourier
coefficient indices $\smallmat{t_1}r{r\trans}{t_2}$ of $\Esixk$ itself.
This multitude of indices gave rise to only $54\,314$ genus
symbols, showing the crucial role of genus coarseness in the
pullback--genus method. Indeed, the indices probably lie in
considerably fewer genera, because we allow genus symbols that are not unique
at~$2$ and the entries of our indices often are divisible by~$2$.
The genera are recorded with multiplicity and the Fourier coefficient
for each genus is computed only once.

Not only do large collections of indices $t$ give rise to far
fewer genus symbols than equivalence classes, but furthermore the
genus symbols are much faster and more space-efficient to compute
because equivalence class computations require a sophisticated
algorithm that uses lattice reduction and maintains an elaborate,
memory-expensive tree structure.
For example, the pair
$$ 
t_1\times t_2
=\left[\begin{matrix}1&\hlf&\hlf\\\hlf&2&1\\\hlf&1&2\end{matrix}\right]
\times\left[\begin{matrix}4&2&2\\2&4&2\\2&2&4\end{matrix}\right]
$$ 
arose in weight~$22$ and gave rise to $6\,755\,849$ semidefinite 
indices~$t=\smallmat{t_1}r{r\trans}{t_2}$.
Some 36 hours of computation on a typical server determined that these
indices fell into $9132$ equivalence classes. On the other hand, only five
minutes of laptop computation produced $4238$ distinct genus symbols
from the indices, and, as in the previous paragraph, the actual number
of genera is smaller.
Another pair $t_1\times t_2$ that we tested separately from our main
computation took about $50$ hours of computation on the server, using
over $1.3$ gigabytes of space, to determine that the resulting
$4\,002\,643$ indices~$t$ fell into $33\,440$ lattice classes, whereas
the laptop computation to produce $9114$ genus symbols from the
indices took only several minutes and under $100$ megabytes of space.

Separately from computing the pulled back Eisenstein series Fourier
coefficients, which are rational, as the weight~$k$ grows so do the
number fields underlying the Hecke eigenforms on the right side
$\sum_\ell c_\ell f_\ell\otimes f_\ell$ of Garrett's formula, and this
posed various programming challenges. The right side summand does not
determine $c_\ell$ or~$f_\ell$ individually, and considerable care was
required to scale them in a way that allowed congruence primes to
be diagnosed. 

%
%



\end{document}